\documentclass[reqno,11pt]{amsart}

\usepackage{amsmath,amssymb,amsthm}
\usepackage{tikz,hyperref}
\usepackage{eucal,mathrsfs,bbm,enumerate}

\usetikzlibrary{shapes,arrows}
\usetikzlibrary{decorations.pathreplacing}
\usetikzlibrary{decorations.markings}

\newtheorem{theorem}{Theorem}
\newtheorem{lemma}[theorem]{Lemma}
\newtheorem{proposition}[theorem]{Proposition}
\newtheorem{corollary}[theorem]{Corollary}

\theoremstyle{definition}

\newtheorem{example}[theorem]{Example}
\newtheorem*{remark}{Remark}

\def\T{\mathscr{T}}
\def\Y{\mathscr{Y}}

\title{A family of Bell transformations}

\author[Birmajer]{Daniel Birmajer}
\address{Nazareth College, 4245 East Ave., Rochester, NY 14618}
\email{abirmaj6@naz.edu}

\author[Gil]{Juan B. Gil}
\author[Weiner]{Michael D. Weiner}

\address{Penn State Altoona, 3000 Ivyside Park, Altoona, PA 16601}
\email{jgil@psu.edu}
\email{mdw8@psu.edu}

\begin{document}

\begin{abstract}
We introduce a family of sequence transformations, defined via partial Bell polynomials, that may be used for a systematic study of a wide variety of problems in enumerative combinatorics. This family includes some of the  transformations listed in the paper by Bernstein \& Sloane, now seen as transformations under the umbrella of partial Bell polynomials. Our goal is to describe these transformations from the algebraic and combinatorial points of view. We provide functional equations satisfied by the generating functions, derive inverse relations, and give a convolution formula. While the full range of applications remains unexplored, in this paper we show a glimpse of the versatility of Bell transformations by discussing the enumeration of several combinatorial configurations, including rational Dyck paths, rooted planar maps, and certain classes of permutations. 
\end{abstract}

\maketitle

\section{Introduction}
\label{sec:intro}

Aiming at developing a unifying approach for a variety of enumeration problems, and in the spirit of the work by E.~T.~Bell on partition polynomials, in this paper we introduce a family of sequence transformations defined via partial Bell polynomials.

Let $a$, $b$, $c$, $d$ be fixed numbers. Given a sequence $x=(x_n)_{n\in\mathbb{N}}$, we let $y=\Y_{a,b,c,d}(x)$ be the sequence defined by
\begin{equation} \label{eq:transform}
   y_n=\sum_{k=1}^n  \frac{1}{n!}
   \bigg[\prod_{j=1}^{k-1}{(an+bk+cj+d)}\bigg] B_{n,k}(1!x_1,2!x_2,\dots) \text{ for } n\ge 1,
\end{equation}
where $B_{n,k}$ denotes the $(n, k)$-th (exponential) partial Bell polynomial. We call $\Y_{a,b,c,d}(x)$ the {\sc Bell} transform of $x$ with parameters $(a,b,c,d)$.

For $k=0,1,2,\dots$, the polynomials $B_{n,k}(z_1,\dots,z_{n-k+1})$ may be defined through the series expansion
\begin{equation*}
 \frac{1}{k!} \bigg(\sum_{j=1}^\infty z_j\, \frac{t^j}{j!}\bigg)^k 
 = \sum_{n=k}^\infty B_{n,k}(z_1,z_2,\dots)\, \frac{t^n}{n!}.
\end{equation*}
These polynomials are homogeneous of degree $k$, of weight $n$, and they can be written as
\begin{equation*}
 B_{n,k}(z_1,\dots, z_{n-k+1})
      = \!\!\sum_{\alpha\in\pi(n,k)}\frac{n!}{\alpha_1!\alpha_2!\cdots}\left(\frac{z_1}{1!}\right)^{\alpha_1}
      \left(\frac{z_2}{2!}\right)^{\alpha_2}\cdots,
\end{equation*}
where $\pi(n,k)$ denotes the set of multi-indices $\alpha\in\mathbb{N}_0^{n-k+1}$ such that
\[  \alpha_1+\alpha_2+\dots=k \;\text{ and }\; \alpha_1+2\alpha_2+3\alpha_3+\dots=n. \]
Note that $B_{n,k}$ contains as many monomials as the number of partitions of $[n]=\{1,\dots,n\}$ into $k$ parts. Thus, if $x$ enumerates some class of building blocks (with $x_j$ distinct blocks of type $j$), then the sequence $\Y_{a,b,c,d}(x)$ counts the number of objects that can be made from these building blocks by placing them (according to their type) on a set of partitions induced by the parameters $(a,b,c,d)$. Moreover, the term
\[ \frac{1}{n!} \bigg[\prod_{j=1}^{k-1}{(an+bk+cj+d)}\bigg] B_{n,k}(1!x_1,2!x_2,\dots) \] 
gives the number of such objects made with exactly $k$ blocks. For example, if the induced partitions consist of interval blocks, then the set of resulting objects of length $n$ made with $k$ such blocks is given by
\begin{equation} \label{eq:compositions}
 \frac{k!}{n!} B_{n,k}(1!x_1,2!x_2,\dots). 
\end{equation}
This corresponds to $(a,b,c,d) = (0,1,-1,1)$. Moreover, the sum over $k$ from $1$ to $n$ gives the known {\sc invert} transform of $x$, see e.g.\ \cite{Bei85,BS95, Cameron89}, and \eqref{eq:compositions} may be interpreted as the number of colored compositions of $n$ with $k$ parts, where part $j$ comes in $x_j$ different colors. 

Another special case of broad interest is the {\sc noncrossing partition} transform, introduced by Beissinger in \cite{Bei85} and systematically studied by Callan in \cite{Callan08}. It corresponds to the parameters $(a,b,c,d) = (1,0,-1,1)$, giving $\prod_{j=1}^{k-1}{(an+bk+cj+d)} = \frac{n!}{(n-k+1)!}$. In this case, \eqref{eq:transform} becomes 
\begin{equation} \label{eq:NCtransform}
   y_n=\sum_{k=1}^n  \frac{1}{(n-k+1)!} B_{n,k}(1!x_1,2!x_2,\dots),
\end{equation}
which counts the configurations obtained by placing the building blocks enumerated by $x$ on top of the noncrossing partitions of $[n]$. In particular, if $x$ is the sequence of ones, denoted by $\mathbbm{1}=(1,1,\dots)$, then 
\begin{align*} \label{eq:NCtransform}
   y_n &=\sum_{k=1}^n  \frac{1}{(n-k+1)!} B_{n,k}(1!,2!,\dots) \\
   &= \sum_{k=1}^n  \frac{1}{(n-k+1)!} \frac{n!}{k!}\binom{n-1}{k-1} \\
   &= \sum_{k=1}^n  \frac{1}{n} \binom{n}{n-k}\binom{n}{k-1} 
   = \frac{1}{n}\binom{2n}{n-1} = \frac{1}{n+1}\binom{2n}{n}.
\end{align*}
Thus $\Y_{1,0,-1,1}(\mathbbm{1})$ is the sequence of Catalan numbers that enumerates noncrossing partitions, Dyck paths, rooted trees, and many others combinatorial objects (see e.g.\ \cite{Sta15}).

The family $\Y_{a,b,c,d}$ also includes several of the known transformations studied by Bernstein and Sloane \cite{BS95}. For example, {\sc exp}, {\sc revert}, and {\sc conv} can be realized as instances of {\sc Bell}. 

Our goal is to study $\Y_{a,b,c,d}$ from the algebraic and combinatorial points of view. In Section~\ref{sec:inverse} we give explicit formulas for the inverse $\Y_{a,b,c,d}^{-1}$. In Section~\ref{sec:gf}, we provide equivalent forms of \eqref{eq:transform} in terms of generating functions. The results in these two sections are obtained using Lagrange inversion together with certain interpolating properties of the partial Bell polynomials proved in \cite{BGW12}. 

In Section~\ref{sec:convolution} we prove a convolution formula for {\sc Bell} transforms of the form $\Y_{a,b,c,1}$, and we give a recurrence relation for $\Y_{a,b,-1,1}$ assuming that $a$ and $b$ are not both equal to $0$. Section~\ref{sec:examples} is a compilation of basic examples that showcase special instances of $\Y_{a,b,c,d}$.
 
In Section~\ref{sec:applications} we discuss combinatorial applications and give some examples that illustrate how $\Y_{a,b,c,d}$ may be used to link the enumeration of certain classes of combinatorial structures with the enumeration of building blocks that serve as ``primitive elements'' within each class. In that section, we will focus on rational Dyck paths, rooted planar maps, and certain classes of permutations. For a general approach to the enumeration of irreducible combinatorial objects, we refer to the work by Beissinger \cite{Bei85}.

It is not surprising that the results of this paper heavily rely on properties of the partial Bell polynomials. For a list of basic properties and combinatorial identities, we refer the reader to \cite{Bell, Charalambides, Comtet, Cvijovic, Eger15, MS08a, MS08b, Mi08, WW09} and the references therein. In an attempt to make the paper more readable and as self-contained as possible, we have included an appendix with the main technical results on partial Bell polynomials used for our arguments.

\section{Inverse relations}
\label{sec:inverse}
We start by proving an interpolating relation for $y=\Y_{a,b,c,d}(x)$, which provides a generalization of Theorem~15 in \cite{BGW12}.
 
\begin{theorem} \label{thm:lambda_general}
Let $x=(x_n)_{n\in\mathbb{N}}$ and let $y=\Y_{a,b,c,d}(x)$, that is,
\begin{equation*}
   y_n=\sum_{k=1}^n  \frac{1}{n!}
   \bigg[\prod_{j=1}^{k-1}(an+bk+cj+d)\bigg] B_{n,k}(1!x_1,2!x_2,\dots) \text{ for } n\ge 1.
\end{equation*}
Assume $c\not=0$. Then, for every $n\in\mathbb{N}$ and for any $\lambda\in\mathbb{C}$, we have
\begin{equation}\label{eq:lambda_general}
\sum_{k=1}^n \bigg[\prod_{j=1}^{k-1}(\lambda-dj+d)\bigg]B_{n,k}(!y)
=\sum_{k=1}^n \bigg[\prod_{j=1}^{k-1}(an+bk+cj+d+\lambda)\bigg] B_{n,k}(!x),
\end{equation}
where $!x=(1!x_1,2!x_2,\dots)$ and\, $!y=(1!y_1,2!y_2,\dots)$.
\end{theorem}
\begin{proof} 
Since $c\not=0$, we can write $y_n$ as
\begin{align*}
   y_n &=\sum_{k=1}^n \frac{1}{n!} \bigg[\prod_{j=1}^{k-1}(an+bk+cj+d)\bigg] B_{n,k}(!x) \\
   &=\sum_{k=1}^n \frac{1}{n!} \bigg[\prod_{j=1}^{k-1}(an+(b+c)k-cj+d)\bigg] B_{n,k}(!x) \\
   &=\sum_{k=1}^n \frac{c^{k-1}}{n!} \binom{\frac{a}{c}n+\frac{(b+c)}{c}k+\frac{d}{c}-1}{k-1} (k-1)! B_{n,k}(!x).
\end{align*}
By the homogeneity property $B_{n,k}(cz_1,cz_2,\dots) = c^k B_{n,k}(z_1,z_2,\dots)$, we then conclude
\begin{equation*}
  n! cy_n = \sum_{k=1}^n \binom{\frac{a}{c}n+\frac{(b+c)}{c}k+\frac{d}{c}-1}{k-1} (k-1)! B_{n,k}(1!cx_1,2!cx_2,\dots).
\end{equation*}

If $d\not=0$, we use Proposition \ref{prop:lambda_c-1} with $x$ replaced by $(1!cx_1,2!cx_2,\dots)$, $y$ replaced by $(1!cy_1,2!cy_2,\dots)$, $\alpha=a/c$, $\beta=(b+c)/c$, $\gamma=d/c$, and $\lambda$ replaced by $\lambda/c$, to obtain
\begin{align*}
\sum_{k=1}^n  \Big(\frac{d}{c}\Big)^{k-1} 
&\binom{\lambda/d}{k-1} (k-1)! B_{n,k}(1!cy_1,2!cy_2,\dots) \\
&= \sum_{k=1}^n\binom{\frac{a}{c} n+\frac{b+c}{c} k +\frac{d}{c} -1 + \frac{\lambda}{c}}{k-1}(k-1)!B_{n,k}(1!cx_1,2!cx_2,\dots) \\
&=\sum_{k=1}^n c\bigg[\prod_{j=1}^{k-1}{(an+bk+cj+d+\lambda)}\bigg] B_{n,k}(1!x_1,2!x_2,\dots).
\end{align*}
This implies
\begin{equation*}
\sum_{k=1}^n  d^{k-1} \tbinom{\lambda/d}{k-1} (k-1)! B_{n,k}(!y)
 =\sum_{k=1}^n \bigg[\prod_{j=1}^{k-1}(an+bk+cj+d+\lambda)\bigg] B_{n,k}(!x),
\end{equation*}
which gives \eqref{eq:lambda_general} noting that $d^{k-1} \binom{\lambda/d}{k-1}(k-1)! = \prod_{j=1}^{k-1}(\lambda-dj+d)$.

If $d=0$, the above expression for $n! cy_n$ reduces to
\begin{equation*}
   n! cy_n = \sum_{k=1}^n \binom{\frac{a}{c}n+\frac{(b+c)}{c}k-1}{k-1} (k-1)! B_{n,k}(1!cx_1,2!cx_2,\dots).
\end{equation*}
Using now Proposition~\ref{prop:lambda_minus1} with $x$ replaced by $(1!cx_1,2!cx_2,\dots)$, $y$ replaced by $(1!cy_1,2!cy_2,\dots)$, $\alpha=a/c$, $\beta=(b+c)/c$, and $\lambda$ replaced by $\lambda/c$, we get
\begin{align*}
\sum_{k=1}^n \Big(\frac{\lambda}{c}\Big)^{k-1} & B_{n,k}(1!cy_1,2!cy_2,\dots) \\
&= \sum_{k=1}^n\tbinom{\frac{a}{c} n+\frac{b+c}{c} k-1 + \frac{\lambda}{c}}{k-1}(k-1)!B_{n,k}(1!cx_1,2!cx_2,\dots) \\
&= \sum_{k=1}^n c\bigg[\prod_{j=1}^{k-1}{(an+bk+cj+\lambda)}\bigg] B_{n,k}(1!x_1,2!x_2,\dots).
\end{align*}
This implies
\begin{equation*}
 \sum_{k=1}^n \lambda^{k-1} B_{n,k}(!y)
 =\sum_{k=1}^n \bigg[\prod_{j=1}^{k-1}(an+bk+cj+\lambda)\bigg] B_{n,k}(!x),
\end{equation*}
which is the statement of the theorem for $d=0$.
\end{proof}

\medskip
As a consequence, we obtain the following inverse relations.

\begin{corollary}\label{cor:inverse}
Let $x=(x_n)_{n\in\mathbb{N}}$ and $y=\Y_{a,b,c,d}(x)$, that is,
\begin{equation*}
   y_n=\sum_{k=1}^n  \frac{1}{n!}
   \bigg[\prod_{j=1}^{k-1}(an+bk+cj+d)\bigg] B_{n,k}(1!x_1,2!x_2,\dots) \text{ for } n\ge 1.
\end{equation*}
Further, define the polynomial $q_{n,k}(t)=t\prod\limits_{j=1}^{k-1}(an+dj+t)$.
\begin{enumerate}[$(i)$]
\item If $c\not=0$, then
  \begin{equation*}
     x_n=\sum_{k=1}^n \frac{(-1)^{k-1}}{n!} \left[\frac{q_{n,k}(b+c)-q_{n,k}(b)}{c}\right] B_{n,k}(1!y_1,2!y_2,\dots).
  \end{equation*}
In other words, 
\begin{equation*}
  \Y_{a,b,c,d}^{-1}=\tfrac{b+c}{c}\Y_{-a,0,-d,-b-c}-\tfrac{b}{c}\Y_{-a,0,-d,-b}. 
\end{equation*}
In particular, $\Y_{a,0,c,d}^{-1}=\Y_{-a,0,-d,-c}$ \,and\, $\Y_{a,-c,c,d}^{-1}=\Y_{-a,0,-d,c}$.
\item If $c=0$, then
  \begin{equation*}
     x_n=\sum_{k=1}^n \frac{(-1)^{k-1}}{n!} q'_{n,k}(b) B_{n,k}(1!y_1,2!y_2,\dots),
  \end{equation*}
where $q'_{n,k}=\frac{d}{dt}q_{n,k}$. In particular, $\Y_{a,0,0,d}^{-1}=\Y_{-a,0,-d,0}$.
\end{enumerate}
\end{corollary}

\begin{proof}
First observe that $(ii)$ follows from $(i)$ by taking the limit as $c\to 0$.

If $c\not=0$, we can apply Theorem~\ref{thm:lambda_general} with $\lambda=-(an+b+c+d)$ to conclude
\begin{align*}
  \sum_{k=1}^n \tfrac{(-1)^{k-1}}{n!} & \tfrac{q_{n,k}(b+c)}{c} B_{n,k}(!y) \\
  &= \sum_{k=1}^n \tfrac{b+c}{c n!}\bigg[\prod\limits_{j=1}^{k-1}(-an-b-c-d-dj+d)\bigg] B_{n,k}(!y) \\
  &= \sum_{k=1}^n \tfrac{b+c}{c n!}\bigg[\prod\limits_{j=1}^{k-1}(b(k-1)+c(j-1))\bigg] B_{n,k}(!x).
\end{align*}
Similarly, but with $\lambda=-(an+b+d)$, we obtain
\begin{equation*}
  \sum_{k=1}^n \tfrac{(-1)^{k-1}}{n!} \tfrac{q_{n,k}(b)}{c} B_{n,k}(!y)
  = \sum_{k=1}^n \tfrac{b}{c n!}\bigg[\prod\limits_{j=1}^{k-1}(b(k-1)+cj)\bigg] B_{n,k}(!x).
\end{equation*}
Finally, for $k\ge 2$ we have
\[ \tfrac{b+c}{c}\prod\limits_{j=1}^{k-1}(b(k-1)+c(j-1)) - \tfrac{b}{c}\prod\limits_{j=1}^{k-1}(b(k-1)+cj) = 0, \]
and therefore,
\begin{align*}
  \sum_{k=1}^n \frac{(-1)^{k-1}}{n!} &\left[\frac{q_{n,k}(b+c)-q_{n,k}(b)}{c}\right] B_{n,k}(!y) \\
  &= \tfrac{b+c}{c n!} B_{n,1}(!x) - \tfrac{b}{c n!} B_{n,1}(!x) = \tfrac{1}{n!} B_{n,1}(!x) = x_n.
\end{align*}
This finishes the proof of $(i)$.
\end{proof}

\section{Generating functions}
\label{sec:gf}

In this section, we will discuss functional equations satisfied by the generating functions of the {\sc Bell} transforms $\Y_{a,b,c,d}$. Our approach to prove these identities showcases some of the interesting properties of partial Bell polynomials. 

As before, we adopt the convenient notation $!z = (1!z_1,2!z_2,\dots)$.

\begin{theorem}\label{thm:gf}
Let $x=(x_n)_{n\in\mathbb{N}}$ and $y=(y_n)_{n\in\mathbb{N}}$ be sequences such that $y=\Y_{a,b,c,1}(x)$ with $c\not=0$. Let $X(t)=\sum\limits_{n=1}^\infty x_n t^n$ and $Y(t)=\sum\limits_{n=1}^\infty y_n t^n$. Then
\begin{equation}\label{eq:gf_d=1}
   X\left(t\big(1+Y(t)\big)^{a}\right) = \frac1c\Big[1-\big(1+Y(t)\big)^{-c}\Big]\big(1+Y(t)\big)^{-b}.
\end{equation}
\end{theorem}

\begin{proof}
Let $Z(t) = 1+Y(t)$. Then the right-hand side of \eqref{eq:gf_d=1} becomes
\begin{equation*}
   \frac1c\Big[1-Z(t)^{-c}\Big]Z(t)^{-b} = \frac1c\Big[Z(t)^{-b} - Z(t)^{-b-c}\Big].
\end{equation*}
By Lemma~\ref{lem:log&power}, we have
\begin{align*}
 n![t^n] Z(t)^{-b} &= -b\sum_{k=1}^n\bigg[\prod_{j=1}^{k-1}{(-b-j)}\bigg] B_{n,k}(!y), \\
 n![t^n] Z(t)^{-b-c} &= -(b+c)\sum_{k=1}^n\bigg[\prod_{j=1}^{k-1}(-b-c-j)\bigg] B_{n,k}(!y).
\end{align*}
Applying \eqref{eq:lambda_general} with $d=1$ and $\lambda=-b-1$, we then get
\begin{equation*}
 n![t^n] Z(t)^{-b} = -b\sum_{k=1}^n \bigg[\prod_{j=1}^{k-1}(an+bk+cj-b)\bigg] B_{n,k}(!x),
\end{equation*}
and similarly, with $d=1$ and $\lambda=-b-c-1$, 
\begin{equation*}
 n![t^n] Z(t)^{-b-c} = -(b+c)\sum_{k=1}^n \bigg[\prod_{j=1}^{k-1}(an+bk+c(j-1)-b)\bigg] B_{n,k}(!x).
\end{equation*}
Now, for $k\ge 2$ we have 
\begin{align*}
 -\tfrac{b}{c} \prod_{j=1}^{k-1} (an + bk 
 &+ cj-b) + \tfrac{b+c}{c} \prod_{j=1}^{k-1}(an+bk+c(j-1)-b) \\
 &= -\tfrac{b}{c} \prod_{j=1}^{k-1}(an+bk+cj-b) + \tfrac{b+c}{c} \prod_{j=0}^{k-2}(an+bk+cj-b) \\
 &= an \prod_{j=1}^{k-2}(an+bk+cj-b) = an \prod_{j=1}^{k-2}(an+b(k-1)+cj),
\end{align*}
and therefore,
\begin{align*}
 n![t^n]\frac{1}{c}\Big(Z(t)^{-b} 
 &- Z(t)^{-b-c}\Big) \\
 &= n! x_n + \sum_{k=2}^n an \bigg[\prod_{j=1}^{k-2}(an+b(k-1)+cj) \bigg] B_{n,k}(!x) \\
 &= n! x_n + an\sum_{k=1}^{n-1} \bigg[\prod_{j=1}^{k-1}(an+bk+cj) \bigg] B_{n,k+1}(!x). 
\end{align*}
Using the identity $B_{n, k+1}(z_1,z_2,\dots)  = \sum\limits_{\ell=k}^{n-1} \binom{n-1}{\ell} z_{n-\ell} B_{\ell, k}(z_1,z_2,\dots)$, we get
\begin{align*} 
 an\sum_{k=1}^{n-1} \bigg[\prod_{j=1}^{k-1}(an+bk 
 &+ cj) \bigg] B_{n,k+1}(!x) \\
 &= a\sum_{k=1}^{n-1} \bigg[\prod_{j=1}^{k-1}(an+bk+cj) \bigg] 
   \sum_{\ell=k}^{n-1} \tfrac{n!(n-\ell)}{\ell!} x_{n-\ell} B_{\ell,k}(!x) \\
 &= a\sum_{\ell=1}^{n-1} \tfrac{n!(n-\ell)}{\ell!} x_{n-\ell}
   \sum_{k=1}^{\ell} \bigg[\prod_{j=1}^{k-1}(an+bk+cj) \bigg] B_{\ell,k}(!x),
\end{align*}
which implies
\begin{equation} \label{eq:KK_RHS}
 [t^n]\frac{1}{c}\Big(Z(t)^{-b} - Z(t)^{-b-c}\Big) 
 = x_n + a\sum_{\ell=1}^{n-1} \tfrac{n-\ell}{\ell!} x_{n-\ell}
 \sum_{k=1}^{\ell} \bigg[\prod_{j=1}^{k-1}(an+bk+cj) \bigg] B_{\ell,k}(!x).
\end{equation}

Let us next examine the left-hand side of \eqref{eq:gf_d=1}. First note that if $a=0$, then \eqref{eq:KK_RHS} gives the claimed statement. If $a\not=0$, we then use Lemma~\ref{lem:log&power} and Identity \eqref{eq:lambda_general} with $d=1$ and $\lambda=a-1$ to conclude
\begin{equation*}
 n![t^n] Z(t)^{a} = a\sum_{k=1}^n \bigg[\prod_{j=1}^{k-1}(an+bk+cj+a)\bigg] B_{n,k}(!x).
\end{equation*}
Now, if $(w_n)_{n\in\mathbb{N}}$ is the sequence defined by $tZ(t)^a = \sum_{n=1}^\infty w_n t^n$, then $w_1=1$ and
\begin{equation*}
 \frac{\ell!w_{\ell+1}}{a} = \sum_{k=1}^{\ell} \bigg[\prod_{j=1}^{k-1}(a\ell+bk+cj+a)\bigg] B_{\ell,k}(!x) \text{ for } \ell\ge 1.
\end{equation*}
Observe that, by means of \eqref{eq:lambda_general} with $d=a$ and $\lambda=an-a\ell-a$, we have
\small
\begin{equation*}
 \sum_{k=1}^\ell \bigg[\prod_{j=1}^{k-1}(an-a\ell-aj)\bigg]B_{\ell,k}(\tfrac{1!w_2}{a},\tfrac{2!w_3}{a},\dots)
 =\sum_{k=1}^\ell \bigg[\prod_{j=1}^{k-1}(an+bk+cj)\bigg] B_{\ell,k}(!x),
\end{equation*}
\normalsize
which implies
\small
\begin{equation} \label{eq:KK_intermediate}
 \sum_{k=1}^\ell \bigg[\prod_{j=1}^{k-1}(n-\ell-j)\bigg]B_{\ell,k}(1!w_2,2!w_3,\dots)
 = a\sum_{k=1}^\ell \bigg[\prod_{j=1}^{k-1}(an+bk+cj)\bigg] B_{\ell,k}(!x).
\end{equation}
\normalsize

On the other hand, Fa{\`a} di Bruno's formula combined with Equation (3$\ell$) in \cite[Sec.~3.3]{Comtet} gives 
\begin{align*}
 n![t]^n X(t Z(t)^a) &= \sum_{\ell=1}^n \ell! x_\ell B_{n,\ell}(1!w_1, 2!w_2,\dots) \\
 &= \sum_{\ell=1}^n \ell! x_\ell \sum_{k=2\ell-n}^\ell \frac{n!}{(n-\ell)! k!} B_{n-\ell,\ell-k}(1!w_2,2!w_3,\dots) \\
 &= \sum_{\ell=1}^n \ell! x_\ell \sum_{k=0}^{n-\ell} \frac{n!}{(n-\ell)! (\ell-k)!} B_{n-\ell,k}(1!w_2,2!w_3,\dots) \\
 &= \sum_{\ell=0}^{n-1} (n-\ell)! x_{n-\ell} \sum_{k=0}^{\ell} \frac{n!}{\ell! (n-\ell-k)!} B_{\ell,k}(1!w_2,2!w_3,\dots).\end{align*}
Since $\ell=0$ gives $n!x_n$, and since $\frac{(n-\ell)!}{(n-\ell-k)!} = (n-\ell)\prod_{j=1}^{k-1}(n-\ell-j)$, we arrive at
\small
\begin{equation} \label{eq:KK_LHS}
 [t^n]X(t Z(t)^a) = x_n + \sum_{\ell=1}^{n-1} \tfrac{n-\ell}{\ell!} x_{n-\ell} 
 \sum_{k=1}^{n-\ell} \bigg[\prod_{j=1}^{k-1}(n-\ell-j) \bigg] B_{\ell,k}(1!w_2,2!w_3,\dots).
\end{equation}
\normalsize
At last, the theorem follows by combining \eqref{eq:KK_RHS}, \eqref{eq:KK_intermediate}, and \eqref{eq:KK_LHS}.
\end{proof}

\begin{corollary}\label{cor:gf}
Let $x=(x_n)_{n\in\mathbb{N}}$ and $y=(y_n)_{n\in\mathbb{N}}$ be sequences such that $y=\Y_{a,b,c,d}(x)$. Let $X(t)=\sum\limits_{n=1}^\infty x_n t^n$ and $Y(t)=\sum\limits_{n=1}^\infty y_n t^n$. Then:

\smallskip
\begin{enumerate}[$(i)$]
\item If $c\not=0$ and $d\not=0$,
\begin{equation*}
 X\left(t\big(1+dY(t)\big)^{a/d}\right)=\frac1c\Big[1-\big(1+dY(t)\big)^{-c/d}\Big]\big(1+dY(t)\big)^{-b/d}.
\end{equation*}
\item If $c=0$ and $d\not=0$,
\begin{equation*}
 X\left(t\big(1+dY(t)\big)^{a/d}\right) = \log\Big(\big(1+dY(t)\big)^{1/d}\Big) \big(1+dY(t)\big)^{-b/d}.
\end{equation*}
\item If $c\not=0$ and $d=0$,
\begin{equation*}
 X\left(te^{aY(t)}\right) = \frac1c\Big[1-e^{-cY(t)}\Big]e^{-bY(t)}.
\end{equation*}
\item If $c=d=0$,
\begin{equation*}
 X\left(te^{aY(t)}\right) = Y(t)e^{-bY(t)}.
\end{equation*}
\end{enumerate}
\end{corollary}
\begin{proof}
Assume $d\not=0$. If we let $\hat x=(dx_n)_{n\in\mathbb{N}}$ and $\hat y=\Y_{\frac{a}{d},\frac{b}{d},\frac{c}{d},1}(\hat x)$, then \eqref{eq:gf_d=1} implies
\begin{equation*}
   \widehat X\left(t\big(1+\widehat Y(t)\big)^{a/d}\right) = \frac{d}{c}\Big[1-\big(1+\widehat Y(t)\big)^{-c/d}\Big]\big(1+\widehat Y(t)\big)^{-b/d}.
\end{equation*}
Using \eqref{eq:transform} together with the homogeneity of the partial Bell polynomials, it is easy to check that $\widehat Y(t) = dY(t)$. Since $\widehat X(t) = dX(t)$, the above identity gives $(i)$. 

The claimed identity in $(ii)$ follows from the identity in $(i)$ by taking the limit as $c\to 0$, and $(iii)$ follows from $(i)$ by taking the limit as $d\to 0$. The statement in $(iv)$ follows from any of the previous items with a corresponding limit argument.
\end{proof}

\section{A convolution formula}
\label{sec:convolution}

In this section we consider a class of sequences, defined through {\sc Bell} transforms, that is amenable to convolutions of arbitrary length.

Once again, we use the notation $!z = (1!z_1,2!z_2,\dots)$.

\begin{proposition}
Let $x=(x_n)_{n\in\mathbb{N}}$ and let $y=\Y_{a,b,c,d}(x)$ with $d\not=0$. Let $\hat y=d\cdot\Y_{a,b,c,d}(x)$, that is, $(\hat y_n)_{n\in\mathbb{N}}$ is the sequence defined by
\begin{equation*}
   \hat y_n = d\cdot \sum_{k=1}^n  \frac{1}{n!}
   \bigg[\prod_{j=1}^{k-1}(an+bk+cj+d)\bigg] B_{n,k}(!x) \text{ for } n\ge 1.
\end{equation*}
If we let $\hat y_0=1$ and $r\in\mathbb{N}$, we have
\begin{equation*}
\sum_{m_1+\dots+m_r=n} \!\! \hat y_{m_1}\cdots \hat y_{m_r} =
  dr\cdot \sum_{k=1}^{n} \frac{1}{n!} \bigg[\prod_{j=1}^{k-1}(an+bk+cj+dr)\bigg] B_{n,k}(!x).
\end{equation*}
\end{proposition}
\begin{proof}
First note that if $Y(t)=\sum\limits_{n=1}^\infty y_n t^n$, then the expression $\sum \hat y_{m_1}\cdots \hat y_{m_r}$ is given by the coefficients of $(1+dY(t))^r$. By Lemma~\ref{lem:log&power}, we have
\begin{align*}
 n![t^n] (1+dY(t))^r &= r\sum_{k=1}^n\bigg[\prod_{j=1}^{k-1}{(r-j)}\bigg] B_{n,k}(!(dy)) \\
 &= r\sum_{k=1}^n d^k\bigg[\prod_{j=1}^{k-1}{(r-j)}\bigg] B_{n,k}(!y) \\
 &= dr\sum_{k=1}^n \bigg[\prod_{j=1}^{k-1}{(dr-dj)}\bigg] B_{n,k}(!y).
\end{align*}
Applying \eqref{eq:lambda_general} with $\lambda=dr-d$, we get
\begin{equation*}
\sum_{k=1}^n \bigg[\prod_{j=1}^{k-1}(dr-dj)\bigg]B_{n,k}(!y)
 =\sum_{k=1}^n \bigg[\prod_{j=1}^{k-1}(an+bk+cj+dr)\bigg] B_{n,k}(!x),
\end{equation*}
which gives the claimed convolution formula.
\end{proof}

\begin{remark}
If $d=1$, we have $\hat y = y = \Y_{a,b,c,1}(x)$. If in addition $c=-1$, we may write the sequence $(y_n)_{n\in\mathbb{N}}$ as
\begin{equation} \label{eq:ab_BellTransform}
   y_n = \sum_{k=1}^{n} \binom{an+bk}{k-1} \frac{(k-1)!}{n!} B_{n,k}(!x),
\end{equation}
and the above convolution formula (with $y_0=1$) becomes
\begin{equation*}
  \sum_{m_1+\dots+m_r=n} \!\! y_{m_1}\cdots y_{m_r} =
  r\cdot\! \sum_{k=1}^{n} \binom{an+bk+r-1}{k-1} \frac{(k-1)!}{n!} B_{n,k}(!x).
\end{equation*}
As discussed by the authors in \cite{BGW14a}, this formula may be used (with appropriate choices of the parameters $a$ and $b$) to obtain convolution formulas of known sequences like Fibonacci, Jacobsthal, Motzkin, Fuss-Catalan, and many others.
\end{remark}

\begin{corollary}
For $a, b\in\mathbb{R}$, not both equal to $0$, the sequence $(y_n)_{n\in\mathbb{N}}$ defined by \eqref{eq:ab_BellTransform} satisfies the recurrence
\begin{equation} \label{eq:abRecurrence}
 y_n=\sum_{\ell=1}^n \; x_\ell \!\! \sum_{i_1+\dots+i_{a\ell+b}=n-\ell}y_{i_1}\dots y_{i_{a\ell+b}}
\end{equation}
where each $i_j$ is a nonnegative integer and $y_0=1$. 
\end{corollary}

A detailed proof of \eqref{eq:abRecurrence} can be found in \cite[Prop.~3.3]{BGMW}.

\section{Examples}
\label{sec:examples}

\begin{example}[{\sc identity}]
If $a=b=c=d=0$, we get $y_n = \frac{1}{n!}B_{n,1}(!x) = x_n$ for every $n\in\mathbb{N}$. Thus $\Y_{0,0,0,0}$ is the identity map.
\end{example}

\begin{example}[{\sc exp}]
If $a=b=c=0$ and $d=1$, then
  \begin{equation*}
     X(t) = \log\big(1+Y(t)\big), \text{ or equivalently, } 1+Y(t) = e^{X(t)}.
  \end{equation*}
Therefore, a sequence $(b_n)$ is the {\sc exp} transform of $(a_n)$, as defined in \cite{BS95}, if and only if $y=(b_n/n!)$ is the {\sc Bell} transform of $x=(a_n/n!)$ with parameters $(0,0,0,1)$.
\end{example}

\begin{example}[{\sc invert}] \label{ex:Bell01}
If $a=b=0$ and $c=d=m$, then
 \begin{equation*}
   X(t)=\frac1m\Big[1-\big(1+mY(t)\big)^{-1}\Big], \text{ or equivalently, } 1+mY(t) = \frac{1}{1-mX(t)}.
 \end{equation*}
For $m\in\mathbb{N}$, $\Y_{0,0,m,m}(x)$ is the $m$-th {\sc invert} transform of $x$.
\end{example}

\begin{example}[{\sc conv}]
If $a=b=0$, $c=-1$, and $d=m$, then
  \begin{equation*}
    X(t) =-1+\big(1+mY(t)\big)^{1/m}, \text{ which implies } (1+X(t))^m = 1+mY(t).
  \end{equation*}
In particular, if $m=2$ and $(x_n)$ is such that $x_0=1$, then $\Y_{0,0,-1,2}=\frac12${\sc conv}.
\end{example}

\begin{example}[{\sc revert}]
The {\sc revert} transform, as it was defined in \cite{BS95}, is closely related to series reversion.
Given a sequence $(a_n)$ with $a_1=1$, let $A(t)=\sum_{n=1}^\infty a_{n}t^n$. The {\sc revert} transform of $(a_n)$ is the sequence $(b_n)$ such that $B(t)=\sum_{n=1}^\infty (-1)^{n+1}b_{n}t^n$ is the power series inverse of $A(t)$. Thus $B(A(t))=t$ and $b_1=1$.

In order to express {\sc revert} in terms of a {\sc Bell} transform, we first recall the following fact (see e.g.\ {\cite[Corollary 11.3]{Charalambides}}):
\begin{equation}\label{eq:series_inversion}
\text{if }\; \phi(t) = t\bigg(1+\sum_{r=1}^\infty \alpha_r \frac{t^r}{r!}\bigg),
\text{ then }\; \phi^{-1}(u) = u\bigg(1+\sum_{n=1}^\infty \beta_n \frac{u^n}{n!}\bigg),
\end{equation}
where
\begin{equation*}
 \beta_n = \sum_{k=1}^n (-1)^k \frac{(n+k)!}{(n+1)!}B_{n,k}(\alpha_1,\alpha_2,\dots).
\end{equation*}
Observe that, since $a_1=1$, we have $A(t)= t\left(1+\sum_{n=1}^\infty a_{n+1}t^n\right)$. So, if we let $\alpha_r = r! a_{r+1}$, the above inversion formula gives
\begin{equation*}
 n!(-1)^{n} b_{n+1} = \sum_{k=1}^n (-1)^k \frac{(n+k)!}{(n+1)!}B_{n,k}(1!a_2,2! a_3,\dots).
\end{equation*}
In other words, for $n\ge 1$,
\begin{equation*}
  b_{n+1}=\sum_{k=1}^n  \frac{1}{n!}
  \bigg[\prod_{j=1}^{k-1}{(-n-j-1)}\bigg] B_{n,k}(1!a_2,-2! a_3,3!a_4,-4!a_5,\dots),
\end{equation*}
which is a {\sc Bell} transform with $a=-1$, $b=0$, $c=d=-1$. Using the operators
\begin{equation}\label{eq:auxOperators}
\begin{aligned}
 L\circ (x_1,x_2,\dots) &= (x_2,x_3,\dots), \\
 R\circ (x_1,x_2,\dots) &= (1,x_1,x_2,\dots), \\
 I\circ (x_1,x_2,\dots) &= (x_1,-x_2,x_3,-x_4,\dots),
\end{aligned}
\end{equation}
we get
\begin{equation*}
\text{\sc revert} = R\circ \Y_{-1,0,-1,-1} \circ  I\circ L.
\end{equation*}

\medskip
Finally, if $X(t)=\sum_{n=1}^\infty (-1)^n a_{n+1} t^n$ and $Y(t)=\sum_{n=1}^\infty b_{n+1} t^n$, then
  \begin{equation*}
      X\left(t(1-Y(t))\right) = -1+\big(1-Y(t)\big)^{-1} = \frac{Y(t)}{1-Y(t)}.
  \end{equation*}
\end{example}

\begin{example}[{\sc polygon dissection}] \label{ex:Bell11}
If $a=1$, $b=0$, $c=d=1$, then
 \begin{equation*}
  X\left(t\big(1+Y(t)\big)\right) = 1-\big(1+Y(t)\big)^{-1}= \frac{Y(t)}{1+Y(t)}.
 \end{equation*}
This transformation, which is clearly related to {\sc revert}, plays a role in the enumeration of polygon dissections by noncrossing diagonals. In particular, if $\mathbbm{1} = (1,1,1,\dots)$ is the sequence of ones, then $\Y_{1,0,1,1}(\mathbbm{1})$ gives the little Schr\"oder numbers $1, 1, 3, 11, 45, 197, 903, 4279, 20793,\dots$, cf.\ \cite[A001003]{Sloane}. A comprehensive discussion on polygon dissections can be found in \cite{BGW17a}.
\end{example}

\begin{example}[{\sc noncrossing partition}] \label{ex:Bell10}
If $a=1$, $b=0$, $c=-1$, $d=1$, then
  \begin{equation*}
    X\left(t\big(1+Y(t)\big)\right)=-\Big[1-\big(1+Y(t)\big)\Big]=Y(t).
  \end{equation*}
This is the functional equation for the {\sc noncrossing partition} transform as defined by Callan in \cite[Eqn.\ (4)]{Callan08}. As already mentioned in the introduction, $\Y_{1,0,-1,1}(\mathbbm{1})$ gives the sequence of Catalan numbers.
\end{example}

\begin{remark}
Observe that, since $\Y_{a,b,c,d} = \Y_{a,b+c,-c,d}$, the {\sc invert} transform is also given by $\Y_{0,1,-1,1}$, and the {\sc polygon dissection} transform is given by $\Y_{1,1,-1,1}$. In other words, the transformations discussed in Examples \ref{ex:Bell01}, \ref{ex:Bell11}, and \ref{ex:Bell10}, are all special instances of the {\sc Bell} transform $\Y_{a,b,-1,1}$. In this case, the elements of the sequence $\Y_{a,b,-1,1}(x)$ may be written as in \eqref{eq:ab_BellTransform}, and the generating functions of $(x_n)_{n\in\mathbb{N}}$ and $(y_n)_{n\in\mathbb{N}}$ satisfy
\begin{equation} \label{eq:ab_BellTransformGF}
 X\left(t\big(1+Y(t)\big)^{a}\right)=Y(t)\big(1+Y(t)\big)^{-b}.
\end{equation}
This particular family of transforms plays a role in the enumeration of colored Dyck paths, colored dissections of convex polygons, colored rooted trees, planar maps, permutations, and more. Some of these applications will be  discussed in the next section.
\end{remark}

We end this section with a brief discussion of the family $\T_{m} \stackrel{\mathrm{def}}{=} \Y_{m,0,-1,1}$.

\begin{example} \label{ex:Bellm0}
For $x=(x_n)_{n\in\mathbb{N}}$ and $m\in\mathbb{Z}$, the sequence $y=\T_{m}(x)$ may be written as
\begin{equation*}
   y_n = \sum_{k=1}^{n} \binom{mn}{k-1} \frac{(k-1)!}{n!} B_{n,k}(1!x_1, 2!x_2, \dots),
\end{equation*}
and in terms of generating functions, we have
\begin{equation*}
  X\left(t\big(1+Y(t)\big)^{m}\right) = Y(t).
\end{equation*}
In particular, if $x=(1,1,\dots)$, then 
\begin{align*}
 y_n &= \sum_{k=1}^n \binom{mn}{k-1}\frac{(k-1)!}{n!} B_{n,k}(1!, 2!, \dots) \\
 &= \sum_{k=1}^n \frac{1}{k} \binom{mn}{k-1} \binom{n-1}{k-1} \\
 &= \sum_{k=1}^n \frac{1}{mn+1} \binom{mn+1}{k}\binom{n-1}{n-k}
 = \frac{1}{mn+1} \binom{(m+1)n}{n}.
\end{align*}
Thus $\T_{m}(\mathbbm{1})$ gives the Fuss-Catalan numbers (see e.g.~\cite{Aval08,HLM08,PS2000}), which for $m\in\mathbb{N}$ enumerate the $m$-divisible noncrossing partitions of $[mn]$, cf.\ \cite{Bei85,Edel80}. 
\end{example}

In general,  $\T_0$ is the identity map, and by Lemma~\ref{lem:lambda} with $\alpha=m$, $\beta=0$, and $\lambda=m'\cdot n$, we obtain
\[ \T_{m'}\circ \T_{m} = \T_{m+m'}, \text{ and consequently, } \T_m^{-1} = \T_{-m}. \] 
This shows that, for $m\in\mathbb{N}$, $\T_{m}(x)$ is the $m$-th {\sc noncrossing partition} transform of $x$.

\bigskip
Here is a connection between the {\sc noncrossing partition} transform and the {\sc binomial} transform, as defined in \cite{BS95}.
\begin{example}
Let $\mathcal S_{\nu}$ be the operator defined by
\[ \mathcal S_{\nu} \circ (x_1,x_2,x_3,\dots) = (x_1+\nu,x_2,x_3,\dots). \]
Given $a=(a_0,a_1,\dots)$, let $b=\text{\sc binomial}(a)$ be the sequence defined by
\[ b_n = \sum_{k=0}^n \binom{n}{k} a_k \;\text{ for } n\in\mathbb{N}_0. \]
Then for $\nu\in\mathbb{N}$, and with $L$ and $R$ as in \eqref{eq:auxOperators}, we have 
\begin{equation*}
  \T_{1}\circ \mathcal S_\nu = L\circ \text{\sc binomial}^\nu\circ R\circ \T_{1}.
\end{equation*}
\end{example}

\section{Application: Combinatorial structures}
\label{sec:applications}

In this section we present a few examples that illustrate how {\sc Bell} transforms may be used to link the enumeration of certain classes of combinatorial structures with the enumeration of suitable building blocks that serve as primitive elements within each class.

\subsection*{Rational Dyck paths}
A class of objects that nicely demonstrates the above principle is the class of rational Dyck paths. Let $\mathcal{D}_{\beta/\alpha}(n)$ be the set of lattice paths from $(0,0)$ to $(\alpha n,\beta n)$ which may touch but never rise above the line $\alpha y=\beta x$, where $n$, $\alpha$, and $\beta$ are positive integers with $\gcd(\alpha,\beta)=1$. Further, let $\mathcal{D}_{\beta/\alpha}^*(n)$ be the set of paths in $\mathcal{D}_{\beta/\alpha}(n)$ that stay {\em strongly} below the line $\alpha y=\beta x$ (except at the endpoints). Clearly, the elements in $\mathcal{D}_{\beta/\alpha}(n)$ that are not in $\mathcal{D}_{\beta/\alpha}^*(n)$ can be thought of as concatenations of shorter paths from the sets $\mathcal{D}_{\beta/\alpha}^*(1), \mathcal{D}_{\beta/\alpha}^*(2),\dots, \mathcal{D}_{\beta/\alpha}^*(n-1)$, so it is not surprising that the sequence $\phi_n= \big|\mathcal{D}_{\beta/\alpha}(n)\big|$ is the {\sc invert} transform of $\psi_n=\big|\mathcal{D}_{\beta/\alpha}^*(n)\big|$. This fact, proved in \cite{Bizley}, implies
\begin{equation*}
   \phi_n = \sum_{k=1}^{n}\frac{k!}{n!}B_{n,k}(1!\psi_1,2!\psi_2,\dots).
\end{equation*}
In other words, with $\phi=(\phi_n)_{n\in\mathbb{N}}$ and $\psi=(\psi_n)_{n\in\mathbb{N}}$, we have
\[ \phi = \Y_{0,1,-1,1}(\psi). \]

However, a different interpretation is possible if we think of rational paths as words over the alphabet $\{a,b\}$. This may be done by associating the letter $a$ with the step $(1,0)$ and the letter $b$ with the step $(0,1)$. 
The elements of $\mathcal{D}_{\beta/\alpha} = \bigcup_{n} \mathcal{D}_{\beta/\alpha}(n)$ then correspond to a generalized Dyck language over $\{a,b\}$ having valuations $h(a)=\beta$ and $h(b)=-\alpha$.
In \cite{Duchon} Duchon established a connection between the set $\mathcal{D}_{\beta/\alpha}$ and its subset of corresponding factor-free words\footnote{A word in a language $L$ is said to be factor-free if it has no proper factor in $L$.}. In fact, if $\theta_n$ denotes the number of factor-free words in $\mathcal{D}_{\beta/\alpha}(n)$, then
\begin{equation}\label{eq:rationalPaths}
  \phi_n = \sum_{k=1}^{n} \binom{(\alpha + \beta)n}{k-1} \frac{(k-1)!}{n!} B_{n, k}(1! \theta_1,2! \theta_2,\dots).
\end{equation}
That is, 
\[ \phi = \Y_{\alpha+\beta,0,-1,1}(\theta). \]
As shown in \cite{BGMW}, equation \eqref{eq:rationalPaths} implies that the elements of $\mathcal{D}_{\beta/\alpha}(n)$ are in bijection to a class of Dyck paths with ascents of length $(\alpha+\beta)j$ (for $j=1,\dots,n$) that may be colored in $\theta_j$ different ways. As a consequence, every path in $\mathcal{D}_{\beta/\alpha}$ may be constructed from factor-free words through an insertion process determined by the corresponding $\theta$-colored Dyck paths. An explicit algorithm for this construction can be found in \cite[Section~2]{BGW17c}.

Finally, we have Bizley's formula (see \cite{Bizley}):
\begin{equation*}
  \phi_n=\sum_{k=1}^{n}\frac1{n!}B_{n,k}(1!f_1,2!f_2,\dots),
\end{equation*}
where $f_j=\frac1{(\alpha+\beta)j}\binom{(\alpha+\beta)j}{\alpha j}$ for $j\in\mathbb{N}$. In other words, 
\[ \phi = \Y_{0,0,0,1}(f), \]
which means that the sequence $(n! \phi_n)_{n\in\mathbb{N}}$ is the {\sc exp} transform of $(n!f_n)_{n\in\mathbb{N}}$. This formula reflects the enumeration of $\mathcal{D}_{\beta/\alpha}(n)$ using cyclical permutations of its elements.

\subsection*{Rooted planar maps}
In \cite{Tutte} Tutte studied the enumeration of rooted planar maps and established a link to the enumeration of nonseparable rooted planar maps, which under the action of $\Y_{2,0,-1,1}$ may be considered their prime elements. 

Let $a_n$ be the number of rooted planar maps with $n$ edges and let $b_n$ be the number of nonseparable rooted planar maps with $n$ edges. Let $A(t)$ and $B(t)$ be the generating functions for $a=(a_n)_{n\in\mathbb{N}}$ \cite[A000168]{Sloane} and $b=(b_n)_{n\in\mathbb{N}}$ \cite[A000139]{Sloane}, respectively. As proved by Tutte (cf.~\cite[Equation~6.3]{Tutte}), these functions satisfy the functional equation 
\begin{equation} \label{eq:planarMaps}
  A(t) = B\big(t(1+A(t))^2\big),
\end{equation}
which by Theorem~\ref{thm:gf} implies
\[ a =  \Y_{2,0,-1,1}(b). \]
In particular, using the results from \cite{BGMW}, we conclude:

\smallskip 
\begin{quote}
There is a bijection between the set of rooted planar maps with $n$ edges and the set of Dyck paths of semilength $2n$ having ascents of even length, and such that each $2j$-ascent may be colored in $b_j$ different ways. 
\end{quote}

\subsubsection*{Bicubic maps} Another example amenable to {\sc Bell} transformations is given by the following connection between rooted bicubic planar maps and their subclass of 3-connected elements. In \cite[Section~11]{Tutte} Tutte observed that ``Each rooted bicubic map can be represented as a multiple extension of a 3-connected bicubic core'' and proved the functional equation
\begin{equation} \label{eq:cubicMaps}
 F(t) = G\big(t(1+F(t))^3\big), 
\end{equation}
where $F(t)=\sum f_n t^n$ enumerates the rooted bicubic maps of $2n$ vertices and $G(t)=\sum g_n t^n$ counts those maps that are 3-connected. This means 
\[ f = \Y_{3,0,-1,1}(g). \]
Therefore, using the results from \cite{BGMW} once again, we conclude: 

\smallskip 
\begin{quote}
There is a bijection between the set of rooted bicubic maps of $2n$ vertices and the set of Dyck paths of semilength $3n$ having ascents of length multiple of 3, and such that each $3j$-ascent may be colored in $g_j$ different ways. 
\end{quote}

Note that the relation $g = \Y_{-3,0,-1,1}(f)$ provides the identity
\begin{equation*}
   g_n = \sum_{k=1}^{n} \binom{-3n}{k-1} \frac{(k-1)!}{n!} B_{n,k}(1!f_1, 2!f_2, \dots).
\end{equation*}
Since $f_j = \frac{3(2j-1)!2^j}{(j-1)!(j+2)!}$ for $j\ge 1$ (see \cite{Tutte} and \cite[A000257]{Sloane}), the above formula\footnote{Which may be written in terms of binomials using Lemma~\ref{lem:bellA257}.} provides a way to compute $g_n$. Here are the first few values: 
\[ 1, 0, 0, 1, 0, 3, 7, 15, 63, 168, 561, 1881, 6110, 21087, \dots \text{ (\cite[A298358]{Sloane}).} \]

\subsubsection*{Eulerian maps} It is well known that rooted bicubic planar maps of $2n$ vertices are in one-to-one correspondence with rooted Eulerian planar maps with $n$ edges. Moreover, as pointed out in \cite{LW04}, the functional equation \eqref{eq:planarMaps} remains valid when restricted to the class of rooted Eulerian planar maps. Thus, if $F(t)$ is as in \eqref{eq:cubicMaps} and $H(t)=\sum h_n t^n$ is the generating function for the nonseparable rooted Eulerian maps with $n$ edges (cf.\ \cite[A069728]{Sloane}), then
\begin{equation} \label{eq:EulerMaps}
  F(t) = H\big(t(1+F(t))^2\big).
\end{equation}
In other words, $f = \Y_{2,0,-1,1}(h)$ and we get the formula
\begin{equation*}
   h_n = \sum_{k=1}^{n} \binom{-2n}{k-1} \frac{(k-1)!}{n!} B_{n,k}(1!f_1, 2!f_2, \dots).
\end{equation*}

\subsection*{Permutations}
Let $S_n$ be the set of permutations on $[n]=\{1,\dots,n\}$, and let $p$ denote the sequence $(n!)_{n\in\mathbb N}$. The inverse of $p$ under the {\sc invert} and {\sc noncrossing partition} transforms, 
\begin{equation}\label{eqn:permTransforms} 
  x_{\textsf{ind}} = \Y_{0,1,-1,1}^{-1}(p) \;\text{ and }\; x_{\textsf{sif}} = \Y_{1,0,-1,1}^{-1}(p), 
\end{equation}
enumerate two known classes of permutations that serve as building blocks for all of the elements of $S_n$. These are sequences A003319 ($x_{\textsf{ind}}$) and A075834 ($x_{\textsf{sif}}$) in \cite{Sloane}. 

\subsubsection*{Stabilized-interval-free permutations}
As introduced by Callan in \cite{Callan04}, a permutation on $[n]$ is said to be stabilized-interval-free (SIF) if it does not stabilize any proper subinterval of $[n]$. The sequence $x_{\textsf{sif}}$ counts the SIF permutations on $[n]$, see \cite[A075834]{Sloane}. 

The fact that $(n!)_{n\in\mathbb N}$ is the {\sc noncrossing partition} transform of $x_{\textsf{sif}}$ was proved and discussed in \cite{Callan04}. As expected, there is a connection to Dyck paths that provides a natural way to visualize the unique decomposition of a permutation on $[n]$ into SIF permutations of length less than or equal to $n$. Specifically, \eqref{eqn:permTransforms} implies that:

\smallskip 
\begin{quote}
The set of permutations on $[n]$ is in one-to-one correspondence with the set of Dyck paths of semilength $n$ such that each $j$-ascent may be colored in as many ways as the number of SIF permutations of length $j$. 
\end{quote}

\subsubsection*{Indecomposable permutations}
The sequence $x_{\textsf{ind}}$ enumerates the class of indecomposable permutations as introduced by Comtet \cite{Comtet72}. This reflects the fact that every permutation on $[n]$ can be split into indecomposable permutations of length less than or equal to $n$, so they play the role of building blocks from which all permutations can be constructed. In fact, every permutation on $[n]$ may be represented as a composition of $n$ whose parts of length $j$ are labeled by indecomposable permutations of length $j$. 

Using this interpretation it is easy to see that, if a pattern $\sigma$ is indecomposable, then every permutation in $\mathsf{Av}(\sigma)$ (set of permutations avoiding the pattern $\sigma$) can be split into $\sigma$-avoiding indecomposable permutations (denoted by $\mathsf{Av}^{\textsf{ind}}(\sigma)$). Thus if $a_{\sigma}=(\mathsf{Av}_n(\sigma))_{n\in\mathbb N}$ and $i_{\sigma}=(\mathsf{Av}_n^{\textsf{ind}}(\sigma))_{n\in\mathbb N}$, then we have (cf.\ \cite[Lem.~3.1]{GKZ16})
\begin{equation} \label{eq:Av_Indecomposable}
  i_{\sigma} = \Y_{0,1,-1,1}^{-1}(a_{\sigma}).
\end{equation}

For example, since the class $\mathsf{Av}(2413)$ is enumerated by \cite[A022558]{Sloane} (cf.\ \cite{Bona97,Stank94}), and since this sequence is the {\sc invert} transform of the right-shifted \cite[A000257]{Sloane}, we conclude 
\begin{equation*}
  \mathsf{Av}_n^{\textsf{ind}}(\sigma) = f_{n-1} = \frac{3(2n-3)!2^{n-1}}{(n-2)!(n+1)!} \;\text{ for } n\ge 2,
\end{equation*}
for every $\sigma\in \{2413, 2431, 3142, 3241, 4132, 4213\}$. In addition, with $f_0=1$ we get
\begin{equation*}
  \mathsf{Av}_n(\sigma) =\sum_{k=1}^{n}\frac{k!}{n!}B_{n,k}(1!f_0,2!f_1,3!f_2,\dots).
\end{equation*}
Similarly, since $\mathsf{Av}(\sigma,\tau)$ for $(\sigma,\tau)\in\{(321, 2341),(321, 3412),(321, 3142)\}$ is enumerated by \cite[A001519]{Sloane}, we can use \eqref{eq:Av_Indecomposable} to obtain 
\begin{equation*}
  \mathsf{Av}_n^{\textsf{ind}}(\sigma,\tau)=2^{n-2} \;\text{ for } n\ge 2. 
\end{equation*}

Finally, if $(\sigma,\tau)$ is any of the following pairs of permutations:
\begin{center}
\begin{tabular}{ccc}
(4321, 4312), & (4312, 4231), & (4312, 4213), \\
(4312, 3412), & (4231, 4213), & (4213, 4132), \\
(4213, 4123), & (4213, 2413), & (3142, 2413),
\end{tabular}
\end{center}
it was shown by Kremer \cite{Kre2000} that $\mathsf{Av}_n(\sigma,\tau)$ is the large Schr\"oder number A006318($n-1$). Hence \eqref{eq:Av_Indecomposable} gives that $\mathsf{Av}_n^{\textsf{ind}}(\sigma,\tau)$ equals the little Schr\"oder number A001003($n-1$).

\subsubsection*{The class $\mathsf{Av}(2413, 3412)$}
{\sc Bell} transformations can be combined with the OEIS \cite{Sloane} to create flows of sequences associated with a particular sequence of interest. This often leads to combinatorial connections that can be verified through the functional equation satisfied by the generating functions. 

Let us illustrate this strategy with the sequence A000257 (number of rooted bicubic maps of $2n$ vertices, number of rooted Eulerian maps with $n$ edges, or $\mathsf{Av}_{n+1}^{\textsf{ind}}(2413)$, for instance). Using {\sc Bell} transformations one discovers the following apparent connections:

\bigskip
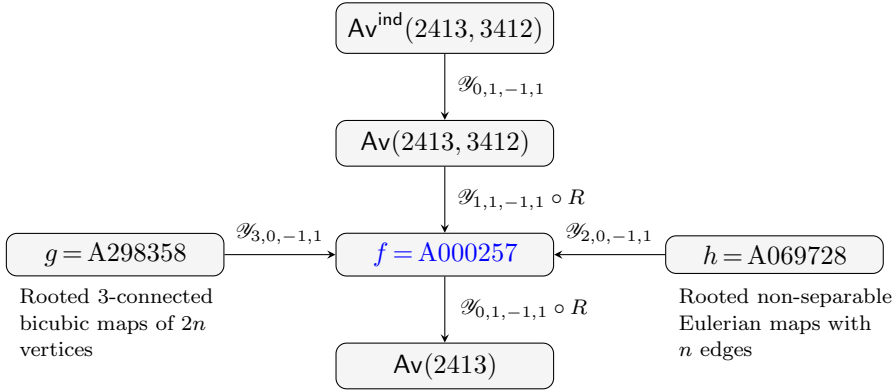
\begin{center}
\tikzstyle{block} = [rectangle, draw, fill=gray!8, text width=75pt, text centered, rounded corners]
\tikzstyle{line} = [draw, -latex']
\usetikzlibrary{positioning}
\begin{tikzpicture}
  \node [block] (C) {{\small \color{blue}$f$\,=\,A000257}};
  \node [block, left = 42pt of C] (L) {{\small $g$\,=\,A298358}};
  \node [block, right = 42pt of C] (R) {{\small $h$\,=\,A069728}};
  \node [block, above = 25pt of C] (U) {{\small $\mathsf{Av}{(2413, 3412)}$}};
  \node [block, above = 25pt of U] (V) {{\small $\mathsf{Av}^{\textsf{ind}}{(2413, 3412)}$}};
  \node [block, below = 25pt of C] (D) {{\small $\mathsf{Av}(2413)$}};
  \path [line,-stealth] (L) -- node[above] {\scriptsize $\Y_{3,0,-1,1}$}(C);
  \path [line,-stealth] (R) -- node[above] {\scriptsize $\Y_{2,0,-1,1}$}(C);
  \path [line,-stealth] (V) -- node[right=1pt] {\scriptsize $\Y_{0,1,-1,1}$}(U);
  \path [line,-stealth] (U) -- node[right=1pt] {\scriptsize $\Y_{1,1,-1,1}\circ R$}(C);
  \path [line,-stealth] (C) -- node[right=2pt] {\scriptsize $\Y_{0,1,-1,1}\circ R$}(D);
  \node [below right=9pt and -40pt] at (R) 
  {\parbox{3cm}{\scriptsize Rooted non-separable\\ Eulerian maps with\\ $n$ edges}};
  \node [below right=9pt and -40pt] at (L) 
  {\parbox{3cm}{\scriptsize Rooted 3-connected\\ bicubic maps of $2n$\\ vertices}};
\end{tikzpicture}
\end{center}
\medskip
where $R$ is the right-shift operator $R\circ (x_1,x_2,\dots) = (1,x_1,x_2,\dots)$. 

On the one hand, the identities 
\[ f=\Y_{3,0,-1,1}(g),\;\; f=\Y_{2,0,-1,1}(h), \text{ and } \mathsf{Av}_n(2413)=\Y_{0,1,-1,1}(R(f)) \] 
are indeed true and consistent with the previously discussed building block approach. On the other hand, the connection between $\mathsf{Av}(2413,3412)$ and $f$ is quite surprising. Note that, since $\Y_{1,1,-1,1}^{-1} = \Y_{-1,0,-1,-1}$, the conjectured identity is equivalent to
\begin{equation}\label{eq:patternAvoidBell}
 R\circ \mathsf{Av}{(2413, 3412)} = \Y_{-1,0,-1,-1}(f).
\end{equation}
If $A(t)$ is the generating function for $\mathsf{Av}{(2413, 3412)}$, then the right-shifted sequence has generating function $t\mathcal{A}(t)$, where $\mathcal{A}(t)=1+A(t)$. Hence \eqref{eq:patternAvoidBell} together with Corollary~\ref{cor:gf} gives the functional equation
\begin{equation*} 
  1 + F\big(t(1-t\mathcal{A}(t))\big) = \frac{1}{1-t\mathcal{A}(t)}.
\end{equation*}
Recall that $F(t)$ is the generating function for the sequence A000257. Now, using the known expression
\[ F(t)=\tfrac{1}{32t^2}\big(-1+12t-24t^2+(1-8t)^{3/2}\big), \] 
and letting $\mathcal{F}(t)=1+F(t)$, one can verify the identity
\begin{equation*} 
16t^2\mathcal{F}(t)^2-(8t^2+12t-1)\mathcal{F}(t)+t^2+11t-1 = 0,
\end{equation*}
which leads to the functional equation
\begin{equation*}
 t^4\mathcal{A}(t)^3 + (5t^3 - 11t^2)\mathcal{A}(t)^2 + (3t^2 + 10t - 1)\mathcal{A}(t) - 9t +1= 0.
\end{equation*}

The validity of this equation was very recently proved by Miner and Pantone \cite{MP18}, thus \eqref{eq:patternAvoidBell} holds. As a consequence, we get the formula
\begin{align*}
   a_{n-1} &= \sum_{k=1}^{n} \binom{-n-2}{k-1} \frac{(k-1)!}{n!} B_{n,k}(1!f_1, 2!f_2, \dots) \text{ for } n\ge 2,
\end{align*}
which may be written in terms of binomials using Lemma~\ref{lem:bellA257}.

\section{Concluding remarks} 
\label{sec:remarks}

The goal of this paper is to introduce a family of {\sc Bell} transformations and to discuss their main properties and combinatorial interpretations. We discussed several instances of these transformations and showcased their unifying principle. The predominant feature has been the use of partial Bell polynomials. As we have shown here, they provide a suitable machinery for dealing with problems that involve Lagrange inversion and to problems that lead to recurrence relations of convolution type. 

We have also shown that, given a class of irreducible/primitive combinatorial objects (building blocks), {\sc Bell} transformations offer an efficient tool for the enumeration of ``composite'' objects constructed from the given set of building blocks. This principle has already been illustrated in Section~\ref{sec:applications}, but it is only appropriate to add a brief discussion about the Bell numbers $(B_n)_{n\in\mathbb{N}}$ (see \cite[A000110]{Sloane}). It is well-known that $B_n$ gives the number of set partitions of $[n]$. They can be built from connected partitions (enumerated by A099947 in \cite{Sloane}) through a standard noncrossing insertion procedure, and if $b$ denotes the sequence $(B_n)_{n\in\mathbb{N}}$, then $b=\Y_{1,0,-1,1}(A099947)$. Alternatively, set partitions can also be built from irreducible partitions (enumerated by A074664 in \cite{Sloane}) through concatenation and shifting, and we have $b=\Y_{0,1,-1,1}(A074664)$. Finally, there is the natural connection via the {\sc exp} transform: $(n!B_n) = \Y_{0,0,0,1}(x)$, where $x_n=\frac{1}{n!}$. 

Such connections via {\sc Bell} transformations can be found effortlessly, and very often they provide a source of combinatorial insight. In addition, the representation in terms of partial Bell polynomials allows for efficient counting, in particular if one is interested in restricting the building blocks used for the construction of the combinatorial objects at hand.

While the applications we considered in this paper are purely combinatorial, it is worth pointing out that partial Bell polynomials and functional equations like the ones in Section~\ref{sec:gf} are certainly relevant in other areas of research, most recently in quantum field theory (see e.g.\ \cite{GSW17,KY17}). The few applications we chose to discuss in Section~\ref{sec:applications} are entirely based on our current interests and are by no means a sign of limitations. We hope that our results serve as motivation to explore other aspects of the {\sc Bell} transformations. Algebraically, it would be interesting to look for and interpret eigensequences (in the spirit of \cite{BS95,Coker04}). Furthermore, it is natural to investigate how the asymptotic behavior of $\Y_{a,b,c,d}(x)$ depends on the asymptotic behavior of $x$, and vice versa.

\section{Appendix: Auxiliary results on partial Bell polynomials}
\label{sec:appendix}

The first two lemmas are known results that may be found in Comtet's book \cite[Sec.~3.4 \& Sec.~3.5]{Comtet}.
\begin{lemma}[Fa\`a di Bruno]\label{lem:FaadiBruno}
Let $f$ and $g$ be two formal power series 
\[ f(u) = \sum_{k=0}^\infty f_k \tfrac{u^k}{k!}, \quad g(t) = \sum_{m=1}^\infty g_m \tfrac{t^m}{m!}, \]
and let $h$ be the formal power series of the composition $f\circ g$,
\[ h(t) = \sum_{n=1}^\infty h_n \tfrac{t^n}{n!} = (f\circ g)(t). \]
Then, the coefficients $h_n$ are given by
\[ h_0=f_0, \quad h_n = \sum_{k=1}^n f_k B_{n,k}(g_1,g_2,\dots,g_{n-k+1}) \text{ for } n\ge 1. \]
\end{lemma}

\begin{lemma}\label{lem:log&power}
The logarithmic polynomials $L_n$, defined by $L_0=1$,
\[  \log(1+g_1t+g_2 \tfrac{t^2}{2!}+\cdots) = \sum_{n=1}^\infty L_n \tfrac{t^n}{n!}, \]
and the potential polynomials $P_n^{(r)}$, defined by $P_0^{(r)}=1$,
\[  \Big(1+g_1t+g_2 \tfrac{t^2}{2!}+\cdots\Big)^r = 1 + \sum_{n=1}^\infty P_n^{(r)} \tfrac{t^n}{n!}, \]
can be written for $n\ge 1$ as
\begin{align*}
 L_n &= \sum_{k=1}^n (-1)^{k-1} (k-1)! B_{n,k}(g_1,g_2,\dots), \\
 P_n^{(r)} &= \sum_{k=1}^n (r)_k B_{n,k}(g_1,g_2,\dots),
\end{align*}
where $(r)_k = \prod_{j=0}^{k-1} (r-j) = \binom{r}{k}k!$. 
\end{lemma}

We now recall a result on partial Bell polynomials, which together with the two extensions below (Propositions \ref{prop:lambda_minus1} and \ref{prop:lambda_c-1}), are needed for the proof of Theorem~\ref{thm:lambda_general}.

\begin{lemma}[{\cite[Theorem~15]{BGW12}}] \label{lem:lambda}
Let $\alpha,\beta \in\mathbb{R}$. Given $x=(x_n)_{n\in\mathbb{N}}$, define $y=(y_n)_{n\in\mathbb{N}}$ by
\begin{equation*}
y_n = \sum_{k=1}^n \tbinom{\alpha n+\beta k}{k-1}(k-1)!B_{n,k}(x).
\end{equation*}
Then, for any $\lambda\in\mathbb{C}$, we have
\begin{equation*}
\sum_{k=1}^n \tbinom{\lambda}{k-1}(k-1)!B_{n,k}(y)
=\sum_{k=1}^n\tbinom{\alpha n+\beta k + \lambda}{k-1}(k-1)!B_{n,k}(x).
\end{equation*}
\end{lemma}

\begin{proposition}\label{prop:lambda_minus1}
Let $\alpha,\beta\in\mathbb{R}$. Given $x=(x_n)_{n\in\mathbb{N}}$, define $y=(y_n)_{n\in\mathbb{N}}$ by
\begin{equation*}
y_n = \sum_{k=1}^n \tbinom{\alpha n+\beta k-1}{k-1}(k-1)!B_{n,k}(x).
\end{equation*}
Then, for any $\lambda\in\mathbb{C}$, we have
\begin{equation}\label{eq:lambda_minus1}
\sum_{k=1}^n \lambda^{k-1} B_{n,k}(y) 
 =\sum_{k=1}^n\tbinom{\alpha n+\beta k-1 + \lambda}{k-1}(k-1)!B_{n,k}(x).
\end{equation}
\end{proposition}
\begin{proof}
Let $f(t)=1+\sum_{n=1}^\infty \lambda^n \frac{t^n}{n!}$ and $g(t)=\sum_{n=1}^\infty y_n \frac{t^n}{n!}$. Then, Lemma~\ref{lem:FaadiBruno} implies
\begin{equation*}
 (f\circ g)(t)=1+\sum_{n=1}^\infty h_n \tfrac{t^n}{n!} \;\text{ with } h_n=\sum_{k=1}^n \lambda^{k} B_{n,k}(y).
\end{equation*}
On the other hand, if we let $z=(z_1,z_2,\dots)$ with
\begin{equation}\label{eq:zn}
 z_n = \sum_{k=1}^n \tbinom{\alpha n+\beta k}{k-1}(k-1)!B_{n,k}(x),
\end{equation}
Lemma~\ref{lem:lambda} implies
\begin{align} \label{eq:z(lambda-1)}
 \sum_{k=1}^n\tbinom{\lambda-1}{k-1}(k-1)!B_{n,k}(z)
 &= \sum_{k=1}^n\tbinom{\alpha n+\beta k + \lambda -1}{k-1}(k-1)!B_{n,k}(x), \\ \label{eq:z(-1)}
 \sum_{k=1}^n\tbinom{-1}{k-1}(k-1)!B_{n,k}(z)
 &= \sum_{k=1}^n\tbinom{\alpha n+\beta k -1}{k-1}(k-1)!B_{n,k}(x).
\end{align}
Note that the left-hand side of \eqref{eq:z(-1)} is the logarithmic polynomial $L_n(z)$. Thus, if we let $Z(t)=1+\sum_{n=1}^\infty z_n \frac{t^n}{n!}$, \eqref{eq:z(-1)} implies $\log(Z(t)) = g(t)$. Hence $(f\circ g)(t) = e^{\lambda g(t)}=Z(t)^\lambda$. 

On the other hand, by Lemma~\ref{lem:log&power}, we have
\begin{equation*}
 \big[\tfrac{t^n}{n!}\big] Z(t)^\lambda = \sum_{k=1}^n (\lambda)_k B_{n,k}(z) 
 = \sum_{k=1}^n \lambda\tbinom{\lambda-1}{k-1}(k-1)!B_{n,k}(z).
\end{equation*}
Therefore, we conclude
\begin{equation*}
 \sum_{k=1}^n \lambda^{k} B_{n,k}(y) = \sum_{k=1}^n \lambda\tbinom{\lambda-1}{k-1}(k-1)!B_{n,k}(z).
\end{equation*}
Combining this identity with \eqref{eq:z(lambda-1)}, we obtain \eqref{eq:lambda_minus1}.
\end{proof}

\begin{proposition} \label{prop:lambda_c-1}
Let $\alpha,\beta,\gamma \in\mathbb{R}$, $\gamma\not=0$. Given $x=(x_n)_{n\in\mathbb{N}}$, define $y=(y_n)_{n\in\mathbb{N}}$ by
\begin{equation*}
y_n = \sum_{k=1}^n \tbinom{\alpha n+\beta k+\gamma-1}{k-1}(k-1)!B_{n,k}(x).
\end{equation*}
Then, for any $\lambda\in\mathbb{C}$, we have
\begin{equation}\label{eq:lambda_c-1}
\sum_{k=1}^n \gamma^{k-1} \tbinom{\lambda/\gamma}{k-1}(k-1)!B_{n,k}(y)
=\sum_{k=1}^n\tbinom{\alpha n+\beta k+\gamma-1 + \lambda}{k-1}(k-1)!B_{n,k}(x).
\end{equation}
\end{proposition}
\begin{proof}
Let 
\[ f(t) = (1+t)^{\frac{\lambda+\gamma}{\gamma}} 
  = 1+\sum_{n=1}^\infty \tbinom{(\lambda+\gamma)/\gamma}{n}n!\frac{t^n}{n!} \] 
and let $g(t)=\sum_{n=1}^\infty \gamma y_n \frac{t^n}{n!}$. Then, by Lemma~\ref{lem:FaadiBruno}, we have
\begin{equation*}
 (f\circ g)(t) = 1+\sum_{n=1}^\infty h_n \tfrac{t^n}{n!} \;\text{ with } h_n
       = \sum_{k=1}^n \tbinom{(\lambda+\gamma)/\gamma}{k}k! B_{n,k}(\gamma y).
\end{equation*}
On the other hand, with $z=(z_n)$ defined as in \eqref{eq:zn}, Lemma~\ref{lem:lambda} implies
\begin{align} \label{eq:z(lambda+c-1)}
 \sum_{k=1}^n\tbinom{\lambda+\gamma-1}{k-1}(k-1)!B_{n,k}(z)
 &= \sum_{k=1}^n\tbinom{\alpha n+\beta k + \lambda+\gamma-1}{k-1}(k-1)!B_{n,k}(x), \\ \label{eq:z(c-1)}
 \sum_{k=1}^n \gamma\tbinom{\gamma-1}{k-1}(k-1)!B_{n,k}(z)
 &= \gamma\sum_{k=1}^n \tbinom{\alpha n+\beta k + \gamma-1}{k-1}(k-1)!B_{n,k}(x).
\end{align}
Since the left-hand side of \eqref{eq:z(c-1)} is the potential polynomial $P^{(\gamma)}_n(z)$, this identity implies $Z(t)^{\gamma}-1 = g(t)$. Finally, since $f(t)=(1+t)^{\frac{\lambda+\gamma}{\gamma}}$, we obtain $(f\circ g)(t) = Z(t)^{\lambda+\gamma}$. Thus
\begin{equation*}
 \sum_{k=1}^n \tbinom{(\lambda+\gamma)/\gamma}{k}k! B_{n,k}(\gamma y) 
 = \sum_{k=1}^n (\lambda+\gamma)_k B_{n,k}(z),
\end{equation*}
which is equivalent to
\begin{equation*}
 \tfrac{\lambda+\gamma}{\gamma}\sum_{k=1}^n \tbinom{\lambda/\gamma}{k-1}(k-1)! B_{n,k}(\gamma y)
 = (\lambda+\gamma)\sum_{k=1}^n \tbinom{\lambda+\gamma-1}{k-1}(k-1)! B_{n,k}(z).
\end{equation*}
Using \eqref{eq:z(lambda+c-1)}, and since $B_{n,k}(\gamma y)=\gamma^kB_{n,k}(y)$, we get \eqref{eq:lambda_c-1}.
\end{proof}

\begin{lemma} \label{lem:bellA257}
Let $(f_j)_{j\in\mathbb{N}}$ be the sequence given by $f_j = \frac{3(2j-1)!2^j}{(j-1)!(j+2)!}$ for $j\ge 1$. For $n,k\ge 1$, we have
\begin{align*}
\tfrac{k!}{n!}B_{n,k}&(1! f_1,2! f_2, 3! f_3,\dots) \\
  &=2^{n+1}\tbinom{2n-1}{n-k}\tfrac{k}{n+k}
  +2^{n+1-2k}\tfrac{k}{n}\sum_{i=0}^k(-1)^i\tbinom{2n+2i-1}{n-1}\tbinom{k-1}{k-i}\\
  &\;\;\;+\sum_{j=1}^{n-1}\sum_{i=1}^{k-1}\sum_{\ell=0}^i(-1)^\ell2^{n+1-2i}\tbinom{k}{i}\tbinom{i-1}{i-\ell}\tbinom{2j+2\ell}{j}\tbinom{2n-2j-1}{n-j-k+i}\tfrac{i(k-i)}{(j+\ell)(n-j+k-i)}.
\end{align*}
\end{lemma}


\end{document}